\documentclass[11pt]{article}%
\usepackage{amsmath}
\usepackage{amsfonts}
\usepackage{amssymb}
\usepackage{graphicx}%
\newtheorem{theorem}{Theorem}

\newtheorem{corollary}[theorem]{Corollary}

\newtheorem{remark}[theorem]{Remark}

\newenvironment{proof}[1][Proof]{\noindent\textbf{#1.} }{\ \rule{0.5em}{0.5em}}
\graphicspath{    {converted_graphics/}    {/}}
\begin{document}

\begin{center}
{\LARGE A Convergence Criterion for the Solutions of Nonlinear}

{\LARGE Difference Equations and Dynamical Systems}\medskip

H. SEDAGHAT \footnote{Department of Mathematics, Virginia Commonwealth
University, Richmond, VA \ 23284, USA;
\par
Email: hsedagha@vcu.edu}
\end{center}

\medskip

\begin{abstract}
A general sufficient condition for the convergence of subsequences of solutions of non-autonomous, nonlinear difference equations and systems is obtained. For higher order equations the delay sizes and patterns play essential roles in determining which subsequences of solutions converge. For systems the specific manner in which the equations are related is important and lead to different criteria. Applications to discrete dynamical systems, including some that model populations of certain species are discussed.

\end{abstract}

\section{Introduction}

Consider the higher-order difference equation or recurrence%
\begin{equation}
x_{n}=F_{n}(x_{n-1},x_{n-2},\ldots,x_{n-m}),\quad n=1,2,3,\ldots\label{g}%
\end{equation}
where the order $m\geq2$ is an integer and $F_{n}:D\rightarrow\mathbb{R}$ for
every $n.$ Here $D$ is a subset of $\mathbb{R}^{m}$ that contains the origin
and is invariant under the standard unfolding
\[
\widehat{F}_{n}(u_{1},u_{2},\ldots,u_{m})=[F_{n}(u_{1},u_{2},\ldots
,u_{m}),u_{1},\ldots,u_{m-1}]
\]
to a system in $\mathbb{R}^{m}$.

Possible disparities among the $m$ variables may lead to convergence only of
certain subsequences of a solution of (\ref{g}) to zero. Such a behavior is
not exceptional; it is indeed observed in certain biological population
models. In \cite{LS1}, Lemma 8, we find this type of behavior exhibited by
interacting adult and juvenile populations; also see Corollary \ref{R} below.
This leads to nontrivial issues regarding the nature of the strong Allee
effect in biological models; see \cite{Cour}, \cite{Cush}, \cite{ES}, \cite{LS},
\cite{LS1}, \cite{LEO}\ for background and further information.

There does not seem to be any systematic research into converging subsequences of 
solutions of difference equations in the existing literature. In this paper we 
study conditions that imply the convergence of certain
subsequences of a solution of (\ref{g}) to zero, including cases where the
entire solution converges; see Theorem \ref{sc}\ and\ Corollary \ref{ct}.
These results also indicate that the size $m$ and the specific pattern of
delay in (\ref{g}) play essential roles in determining which subsequences converge.

This type of subsequence convergence is manifested in nonlinear equations and
systems in dimensions two and greater, since both nonlinearity and at least
two variables are required (see conditions H1-H3 below) for the existence of a
disparity of the type where a proper subsequence of a solution may converge.
Therefore, we do not expect to observe this type of behavior in first-order
equations, or in linear systems in any dimension.

The results on higher order equations also apply to systems by folding the
system to a scalar difference equation (of order two or greater; see \cite{HS}). In
particular, we discuss the convergence of subsequences of solutions of certain
planar systems that are used in modeling biological population dynamics. We
obtain sufficient conditions on the system's equations directly using
additional system-specific conditions H5 and H6, thus avoiding explicit
folding calculations in specific problems.

\section{Convergence in higher order equations}

We assume that the functions $F_{n}$ satisfy the following hypotheses:

\begin{quote}
\textbf{H1}. There is a non-negative real function $g$ and $k\in\left\{
1,2,\ldots, m\right\}  $ such that $\left\vert F_{n}(u_{1},\ldots,
u_{m})\right\vert \leq g(u_{k})$ for all $n\geq1$ and all $(u_{1},\ldots,
u_{m})\in D$;

\textbf{H2}. There is $\alpha\in(0,\infty]$ such that $g(u)<|u|$ for all
$u\in(-\alpha,\alpha)\cap\pi_{k}(D)$, $u\not =0$ where $\pi_{k}(D)$ is the
projection of $D$ onto the $k$-th axis, i.e.
\[
\pi_{k}(D)=\{u_{k}:(u_{1},\ldots, u_{k},\ldots, u_{m})\in D\}
\]

\textbf{H3}. The function $g(u)$ is continuous on $(-\alpha,\alpha)\cap\pi
_{k}(D)$.
\end{quote}

Note that the functions $F_{n}$ need not be continuous on $D$ and the number
$\alpha$ may be arbitrarily large. However, H1-H3 above imply that $g(0)=0$
and $F_{n}(0,\ldots,0)=0.$

Hypotheses H1-H3 bound the maps $F_{n}$ in a neighborhood of the origin but
these maps are not otherwise restricted. They need not be monotone, continuous
or have any other specialized properties near the origin. Rather, it is the
map $g$ that, in some neighborhood of the origin, needs to be sublinear by H2
and continuous by H3. We also point out that H1-H3 do not rule out the
possibility that $k=m=1$ (i.e. (\ref{g}) has order 1). In this case, we show
that convergence occurs for the entire solution rather than a proper
subsequence of it. Hence, when the order is 1 the subseqeunce convergence does
not occur properly.

The next result shows that H1-H3 imply the convergence of the subsequences
associated with $u_{k}$ of any solution of (\ref{g}) that comes within the
$\alpha$-distance of the origin.

\begin{theorem}
\label{sc}Assume that $F_{n}$ satisfy H1-H3 and let $\{x_{n}\}$ be a solution
of (\ref{g}). If there is an integer $n_{0}\geq1$ such that $x_{n_{0}}%
\in(-\alpha,\alpha)\cap\pi_{k}(D)$ then $\lim_{j\rightarrow\infty}x_{n_{0}%
+kj}=0.$
\end{theorem}

\begin{proof}
It is convenient to work with a symmetric modification of $g$ so we define
\[
h(u)=\max\{g(u),g(-u)\}
\]
for every $u\in(-\alpha,\alpha)\cap\pi_{k}(D)$. This is an even function since
$h(-u)=h(u)$. Further, $h$ is continuous on $(-\alpha,\alpha)\cap\pi_{k}(D)$,
$h(0)=0$ and%
\[
g(u)\leq h(u)<|u|
\]
for all $u\in(-\alpha,\alpha)\cap\pi_{k}(D).$ Because%
\[
|x_{n_{0}+k}|=|F_{n_{0}+k}(x_{n_{0}+k-1},x_{n_{0}+k-2},\ldots,x_{n_{0}%
},x_{n_{0}-1}\ldots,x_{n_{0}+k-m})|\leq g(x_{n_{0}})\leq h(x_{n_{0}})
\]
it follows that
\[
|x_{n_{0}+k}|\leq h(x_{n_{0}})<|x_{n_{0}}|<\alpha.
\]

These inequalities continue to similarly hold for $x_{n_{0}+2k}$ etc, yielding%
\begin{equation}
|x_{n_{0}+(j+1)k}|\leq h(x_{n_{0}+jk})<|x_{n_{0}+jk}|\leq\cdots\leq
h(x_{n_{0}+k})<|x_{n_{0}+k}|\leq h(x_{n_{0}})<|x_{n_{0}}|<\alpha. \label{ch}%
\end{equation}

In particular, the sequence $|x_{n_{0}+jk}|$ decreases strictly as $j$
increases so if $\zeta=\inf_{j\geq1}|x_{n_{0}+jk}|>0$ then $h(\zeta)<\zeta$
since $\zeta<\alpha.$ On the other hand, (\ref{ch}) implies that
$\lim_{j\rightarrow\infty}h(x_{n_{0}+jk})=\zeta$ so $h$ being continuous and
$h(|u|)=h(u)$, it follows that $h(\zeta)=\zeta.$ But this contradicts the
inequality $h(\zeta)<\zeta$ unless $\zeta=0$.
\end{proof}

\medskip

As an immediate corollary we have the following result that also illustrates
the special role of the variable with the least time delay in the convergence
of the whole solution to zero.

\begin{corollary}
\label{ct}Assume that $F_{n}$ and $g$ are as in Theorem \ref{sc} and let
$\{x_{n}\}$ be a solution of (\ref{g}). If there is $n_{0}\geq0$ such that
$x_{n_{0}+i}\in(-\alpha,\alpha)\cap\pi_{k}(D)$ for $i=0,1,\ldots, k-1$ then
$\lim_{n\rightarrow\infty}x_{n}=0$. In particular, if $k=1$ and $x_{n_{0}}%
\in(-\alpha,\alpha)\cap\pi_{k}(D)$ then $\lim_{n\rightarrow\infty}x_{n}=0$.
\end{corollary}

Theorem \ref{sc} applies to large classes of difference equations and systems.
To illustrate, consider the following higher-order equation a special case of
which is a biological population model for certain species (see below):%
\begin{equation}
x_{n}=x_{n-k}^{\lambda}e^{a_{n}-b_{1,n}x_{n-1}-b_{2,n}x_{n-2}-\cdots
-b_{m,n}x_{n-m}} \label{spm}%
\end{equation}
where%
\begin{equation}
\lambda>1,\quad k\leq m,\quad a_{n}\leq a<\infty,\quad b_{i,n}\geq
0,\ b_{k,n}\geq b>0 \label{cnst}%
\end{equation}
for all $n\geq1$ and $i=1,2,\ldots,m.$ We assume that $D=[0,\infty)^{2}$ and
consider only the non-negative solutions of (\ref{spm}).

A straightforward calculation verifies that the equation $u^{\lambda}%
e^{a-bu}=u$ has roots $u^{\ast}$ and $\bar{u}$ where $0<u^{\ast}\leq\bar{u}$
provided that
\begin{equation}
a\geq(\lambda-1)[1+\ln b-\ln(\lambda-1)]. \label{lam}%
\end{equation}

\begin{corollary}
\label{R}Assume that (\ref{cnst}) and (\ref{lam}) hold. If $\{x_{n}\}$ is a
positive solution of (\ref{spm}) and $x_{n_{0}}\in(0,u^{\ast})$ for some
$n_{0}\geq1$ then the terms $x_{n_{0}+jk}$ for $j=0,1,\ldots$ decrease
(strictly) to 0. If $k=1$ then the solution $\{x_{n}\}$ converges
monotonically to 0.
\end{corollary}

\begin{proof}
In this case $g(u)=h(u)$ since $u\geq0$. If $F_{n}(u_{1},\ldots,u_{m}%
)=u_{k}^{\lambda}e^{a_{n}-b_{k,n}u_{k}-\sum_{i\not =k}b_{i,n}u_{i}}$ then for
all $n$
\[
F_{n}(u_{1},\ldots,u_{m})\leq u_{k}^{\lambda}e^{a_{n}-b_{k,n}u_{k}}\leq
u_{k}^{\lambda}e^{a-bu_{k}}.
\]

With $g(u)=u^{\lambda}e^{a-bu}$ and $\alpha=u^{\ast}$ Theorem \ref{sc}
completes the proof of convergence to 0. The monotone nature of the converging
subsequence $\{x_{n_{0}+jk}\}$ is obvious from the proof of Theorem \ref{sc}
since all absolute values may be removed for positive solutions. The last
statement is now clear.
\end{proof}

Figures 1-4 illustrate the preceding results for the following third-order,
autonomous special case of (\ref{spm})%
\begin{equation}
x_{n}=x_{n-k}^{3/2}e^{3/2-0.7x_{n-2}-0.9x_{n-3}} \label{sp3}%
\end{equation}

All of the solutions of (\ref{sp3}) shown in the figures have the same initial
values $x_{0}=x_{1}=x_{2}=1.$ The only thing that is different in each figure
is the value of $k$.

\begin{figure}[tbp] 
  \centering
  \includegraphics[width=3.83in,height=1.99in,keepaspectratio]{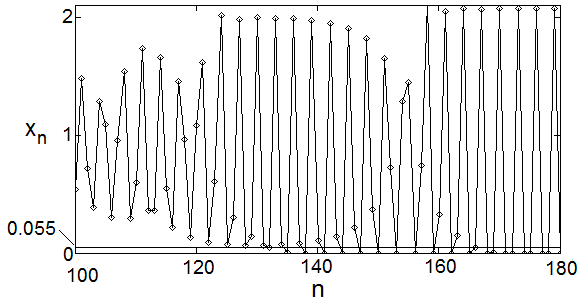}
  \caption{A solution with two subsequences converging to zero ($k=3$)}
  \label{Fig1}
\end{figure}

In Figure \ref{Fig1} where $k=3$, we see two subsequences converge to zero. At
$n=132$ one term of the solution crosses the threshold below which $g(u)<u$
for $u>0$ where $g(u)=u^{3/2}e^{3/2-0.9u}$; see Figure \ref{Fig2} (this
threshold is known as the Allee threshold in biological contexts; see below
for more details and references). Thus, according to Corollary \ref{R} the
subsequence $\{x_{132+3j}\}=\{x_{132},x_{135},x_{138},\ldots\}$ converges
monotonically to zero, as we also see in Figure \ref{Fig1}.

\begin{figure}[tbp] 
  \centering
  \includegraphics[width=2.27in,height=1.99in,keepaspectratio]{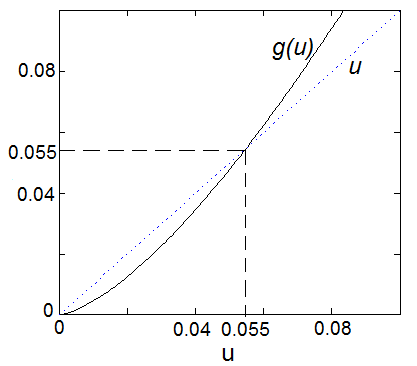}
  \caption{The bounding map $g$ and the threshold ($k=3$)}
  \label{Fig2}
\end{figure}

Corollary \ref{R} also permits additional subsequences to converge to zero and
this is seen in Figure \ref{Fig1} also. Specifically, at $n=166,$ which is not
an integer of type $132+3j,$ a term of the solution crosses the threshold so
that the subsequence $\{x_{166+3j}\}=\{x_{166},x_{169},x_{172},\ldots\}$ also
converges to zero monotonically. The remaining terms $\left\{  x_{167+3j}%
\right\}  =\{x_{167},x_{170},x_{173},\ldots\}$ converge in this case to a
positive fixed point of $g$. This may be due to the fact that for large $n$ 
and $k=3$ the terms $x_{165+3j}$ and $x_{166+3j}$ are nearly zeros, so
the remaining terms nearly satisfy the reduced expression on the right hand side
of (\ref{sp3}), i.e.
$$x_{n-3}^{3/2}e^{3/2-0.9x_{n-3}}$$
It follows that every third term is indistinguishable from the iterates of $g$
that start farther than 0.055 from the origin.

\begin{figure}[tbp] 
  \centering
  \includegraphics[width=3.57in,height=1.99in,keepaspectratio]{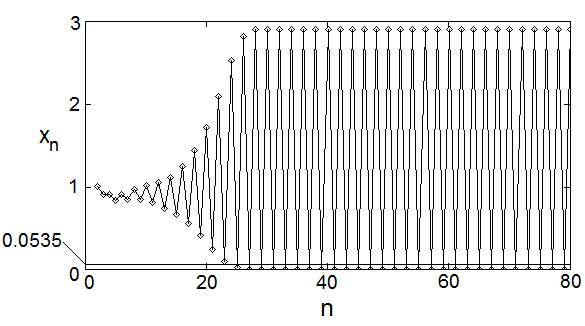}
  \caption{A solution with one subsequence converging to zero ($k=2$)}
  \label{Fig3}
\end{figure}

Figure \ref{Fig3} shows the solution that is generated when $k=2$, all other
parameters being the same. At $n=25$ a term of the solution crosses the
threshold below which $g(u)<u$ for $u>0$ where $g(u)=u^{3/2}e^{3/2-0.7u}$;
Thus, according to Corollary \ref{R} the subsequence $\{x_{25+2j}%
\}=\{x_{25},x_{27},x_{29},\ldots\}$ converges monotonically to zero, as is
also seen in Figure \ref{Fig3}.

\begin{figure}[tbp] 
  \centering
  \includegraphics[width=3.65in,height=1.99in,keepaspectratio]{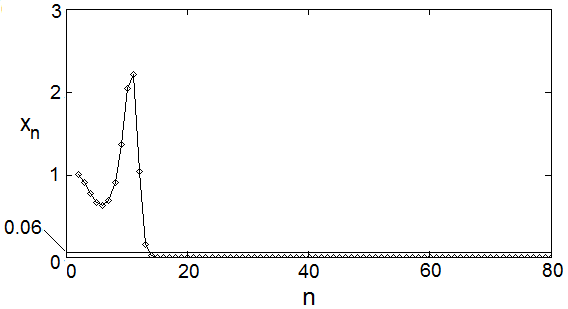}
  \caption{A solution that converges to zero ($k=1$)}
  \label{Fig4}
\end{figure}

In Figure \ref{Fig4} where $k=1$ we see that as soon as any term of the
solution crosses the threshold ($n=14$ in the figure) below which $g(u)<u$ for
$u>0$ where $g(u)=u^{3/2}e^{3/2-(0.7+0.9)u}$, the entire solution converges
monotonically to zero. This is also consistent with both Corollary \ref{ct}
and Corollary \ref{R}.

A special case of (\ref{spm}) appears in a biological population model that
exhibits the strong Allee effect; see \cite{Cour}, \cite{Cush}, \cite{ES},
\cite{LS}, \cite{LEO}\ for background and further information on this and
related models. Consider the planar system%
\begin{align}
x_{n+1}  &  =s_{n}y_{n}\label{aj1}\\
y_{n+1}  &  =x_{n}^{\lambda}e^{r_{n}-x_{n}-t_{n}y_{n}} \label{aj2}%
\end{align}
where $s_{n}\in(0,1],$ $t_{n}\in(0,\infty),$ $r_{n}\in(-\infty,\infty)$ are
given bounded sequences representing biological parameters, and $x_{n}$ and
$y_{n}$ denote the populations (or densities) of adults and juveniles,
respectively. In this model, all adults are removed from the population at the
end of each period either by natural death (as in the case of a semelparous
species, i.e. all adults die in each generation, and juveniles become adults
in the next) or through harvesting, predation, migration, etc. The system
(\ref{aj1})-(\ref{aj2}) folds to a second-order difference equation as
follows:%
\[
x_{n+2}=s_{n+1}y_{n+1}=s_{n+1}x_{n}^{\lambda}e^{r_{n}-x_{n}-t_{n}y_{n}}%
=x_{n}^{\lambda}e^{r_{n}+\ln s_{n+1}-(t_{n}/s_{n})x_{n+1}-x_{n}}%
\]

This equation is a special case of (\ref{spm}) as it can be stated
equivalently in the form%
\begin{equation}
x_{n}=x_{n-2}^{\lambda}e^{r_{n-2}+\ln s_{n-1}-(t_{n-2}/s_{n-2})x_{n-1}%
-x_{n-2}} \label{aj}%
\end{equation}

From initial values $x_{0}$ and $y_{0}$, we obtain $x_{1}=s_{0}y_{0}$ form
(\ref{aj1}). A solution $\{x_{n}\}$ of (\ref{aj}) is thus generated that
determines the adult population for $n\geq2$. The juvenile population is then
given by $y_{n}=x_{n+1}/s_{n}.$

Since (\ref{aj}) is a special case of (\ref{spm}), Corollary \ref{R} implies
that if $x_{n_{0}}\in(0,u^{\ast})$ for some even (or odd) $n_{0}\geq1$ then
the even (respectively, odd) terms of $\{x_{n}\}$ converge to 0 (also see
Corollary \ref{syst} below). This is consistent with the fact that adults are
absent every other period. In particular, if $x_{0}>0$ and $y_{0}=0$ (no
juveniles in the initial mix) then the solution starts oscillating from axis
to axis rather than converging to an oscillatory solution. But if the initial
mix contains some juveniles ($y_{0}>0$)\ then the orbit converges to a
solution that oscillates axis to axis. It is worth emphasizing that this
oscillation may not be periodic even with constant coefficients.

Theorem \ref{sc} applies to fixed points other than the origin via
translations, provided that these fixed points are not unstable or repelling.
To illustrate, consider the following equation%
\begin{equation}
x_{n}=\frac{a_{n}(x_{n-k}-b)^{p}}{1+c_{n}x_{n-l}^{q_{n}}}+b \label{n0}%
\end{equation}
where $0<a_{n}\leq a$ for some $a>0$ and $c_{n},b\geq0$ for all $n$. Further,
we assume that $q_{n}>0$ and to ensure that the right hand side of (\ref{n0})
returns real values, $p$ is a rational number of type $2\gamma/(2\delta-1)$
where $\gamma,\delta$ are positive integers. Note that (\ref{n0}) has a fixed
point at $\bar{x}=b.$ The special case of (\ref{n0}),%
\[
x_{n}=\frac{a_{n}x_{n-k}^{p}}{1+c_{n}x_{n-k}^{p}}%
\]
may be considered a version of the sigmoid Beverton-Holt equation in
\cite{HKK} that contains a delay $k\geq1$.

If $m=\max\{k,l\}$ then the underlying map sequence of (\ref{n0}) is%
\[
F_{n}(u_{1},u_{2},\ldots, u_{m})=\frac{a_{n}(u_{k}-b)^{p}}{1+c_{n}u_{l}%
^{q_{n}}}+b
\]
which unfolds to the map sequence%
\[
\widehat{F}_{n}(u_{1},u_{2},\ldots, u_{m})=\left[  \frac{a_{n}(u_{k}-b)^{p}%
}{1+c_{n}u_{l}^{q_{n}}}+b,u_{1},\ldots, u_{m-1}\right]
\]

Each $\widehat{F}_{n}$ has an invariant set $D=[0,\infty)^{m}$ that contains
the fixed point $(b,\ldots, b)$ in its interior if $b>0$. To apply Theorem
\ref{sc} we shift $\widehat{F}_{n}$ by translation so that the fixed point is
at origin$.$ We obtain%
\[
\widetilde{F}_{n}(u_{1},\ldots, u_{m})=\widehat{F}_{n}(u_{1}+b,\ldots,
u_{m}+b)-(b,\ldots, b)=\left[  \frac{a_{n}u_{k}^{p}}{1+c_{n}(u_{l}+b)^{q_{n}}%
},u_{1},\ldots, u_{m-1}\right]
\]

The invariant set for each $\widetilde{F}_{n}$ is $[-b,\infty)^{m}$ on which
we may fold $\widetilde{F}_{n}$ to
\[
F_{n}^{\ast}(u_{1},\ldots, u_{m})=\frac{a_{n}u_{k}^{p}}{1+c_{n}(u_{l}%
+b)^{q_{n}}}%
\]

Note that $F_{n}^{\ast}$ is defined on $[-b,\infty)$. Since $\left\vert
F_{n}^{\ast}(u_{1},\ldots, u_{m})\right\vert \leq a|u|_{k}^{p}$ if we define%
\[
g(u)=a|u|^{p},\quad u\geq-b
\]
then $\left\vert F_{n}^{\ast}(u_{1},\ldots, u_{m})\right\vert \leq g(u_{k})$.
Next, $g(u)<|u|$ for $u\not =0$ if and only if%
\[
|u|<a^{-1/(p-1)}%
\]

Let $\alpha=a^{-1/(p-1)}>0$. With $F_{n}^{\ast}$ defining the difference
equation
\[
y_{n}=\frac{a_{n}y_{n-k}^{p}}{1+c_{n}(y_{n-l}+b)^{q_{n}}}%
\]
Theorem \ref{sc} implies that if there is $n_{0}\geq1$ such that if $y_{n_{0}%
}\in(\max\{-\alpha,-b\},\alpha)$ then $\lim_{j\rightarrow\infty}y_{n_{0}%
+jk}=0$. Note that
\[
(\max\{-\alpha,-b\},\alpha)=(-\alpha,\alpha)\cap\pi_{k}([-b,\infty)^{m}).
\]

Since $x_{n}=b+y_{n}$ the following result is established.

\begin{corollary}
If $\{x_{n}\}$ is a solution of (\ref{n0}) such that $x_{0},\ldots, x_{m-1}>0$
and $x_{n_{0}}\in(\max\{0,b-a^{-1/(p-1)}\},b+a^{-1/(p-1)})$ for some
$n_{0}\geq1$ then $\lim_{j\rightarrow\infty}x_{n_{0}+jk}=b.$
\end{corollary}

If (\ref{n0}) has additional fixed points then we may apply the preceding
ideas to those other fixed points provided that they are not unstable. Also
worth mentioning, if $q_{n}$ is the same type of rational as $p$ for every $n$
then the function on the right hand side of (\ref{n0}) is defined on the
entire space $\mathbb{R}^{m}$ because the denominator remains positive. In
this case, $\{a_{n}\}$ maybe any bounded sequence in $\mathbb{R}$ as long as
$c_{n}\geq0$ for all $n.$ The preceding corollary is applicable in this
extended domain.

\section{Planar systems and population models}

Theorem \ref{sc} may be applied to discrete systems. The basic idea is to fold
the system to a higher-order equation as we showed above for the system
(\ref{aj1})-(\ref{aj2}), although we see in this section that actual folding
calculations are often unnecessary. Consider the planar system
\begin{align}
x_{n+1}  &  =f_{n}(x_{n},y_{n})\label{s11}\\
y_{n+1}  &  =g_{n}(x_{n},y_{n}) \label{s12}%
\end{align}
where $f_{n},g_{n}:D\rightarrow\lbrack0,\infty)$ are given sequences of
functions on a domain $D\subset\lbrack0,\infty)^{2}$ that is invariant for the
system. The system (\ref{aj1})-(\ref{aj2}) is of this type. From an initial
point $(x_{0},y_{0})\in D$ the correspondng solution of the system
(\ref{s11})-(\ref{s12}) is an orbit $\{(x_{n},y_{n})\}$ in the quadrant
$[0,\infty)^{2}$. To ensure that the origin is a fixed point of the system,
assume that for all $n$%
\begin{equation}
f_{n}(0,0)=g_{n}(0,0)=0. \label{o}%
\end{equation}

Applications of Theorem \ref{sc} to systems may occur in two different ways,
depending on the way the functions $f_{n},g_{n}$ relate to each other. Both
types of systems appear among biological population models. We distinguish
between these two types of systems via two different corollaries of Theorem
\ref{sc}.

Although folding calculations are often not necessary for our purposes in this
section, in principle folding is required in order to use Theorem \ref{sc}. So
we assume the following (see \cite{HS} for further details)

\begin{quote}
\textbf{H4.}\textit{ Each function }$f_{n}(u,v)$\textit{ is solvable for }%
$v$\textit{, i.e. there is a sequence of functions }$\sigma_{n}:[0,\infty
)^{2}\rightarrow\lbrack0,\infty)$\textit{ such that}
\[
w=f_{n}(u,v)\Rightarrow v=\sigma_{n}(u,w).
\]

\end{quote}

Alternatively, we might require $g_{n}(u,v)$ to be solvable for $u$ in an
analogous sense, if more feasible. Under suitable differentiability
hypotheses, the functions $\sigma_{n}$ can be shown to exist locally using the
implicit function theorem. However, in many applications, including in some
population models, (e.g. the system (\ref{rbh1})-(\ref{rbh2}) below) the
functions $\sigma_{n}$ can be obtained analytically by routine calculation.
Also in many cases, \textit{separable functions} of type%

\[
f_{n}(u,v)=\rho_{n}(u)\phi_{n}(v)\quad\text{or\quad}f_{n}(u,v)=\rho
_{n}(u)+\phi_{n}(v)
\]
appear where $\phi_{n}:[0,\infty)\rightarrow\lbrack0,\infty)$ is a bijection
and $\rho_{n}:[0,\infty)\rightarrow(0,\infty)$ for each $n$. In these cases
each $f_{n}$ is \textit{globally} solvable and we obtain the explicit
expressions
\[
\sigma_{n}(u,w)=\phi_{n}^{-1}\left(  \frac{w}{\rho_{n}(u)}\right)
\quad\text{or\quad}\sigma_{n}(u,w)=\phi_{n}^{-1}(w-\rho_{n}(u))
\]

For instance, the multiplicatively separable type occurs in (\ref{aj1}).
Population models where one equation in the system is multiplicatively or
additively separable frequently appear in the literature; see, e.g. \cite{AJ},
\cite{Cush2}, \cite{LS}, \cite{LP}.

Assuming H4 we obtain from (\ref{s11})%
\begin{equation}
y_{n}=\sigma_{n}\left(  x_{n},x_{n+1}\right)  \label{yn}%
\end{equation}

Using this relation and (\ref{s11})-(\ref{s12}) we obtain the second-order
scalar equation%
\[
x_{n+1}=f_{n}(x_{n},g_{n-1}(x_{n-1},y_{n-1}))=f_{n}(x_{n},g_{n-1}\left(
x_{n-1},\sigma_{n-1}\left(  x_{n-1},x_{n}\right)  \right)  )
\]
or equivalently, after an index shift,%
\begin{equation}
x_{n}=f_{n-1}(x_{n-1},g_{n-2}\left(  x_{n-2},\sigma_{n-2}\left(
x_{n-2},x_{n-1}\right)  \right)  ) \label{de}%
\end{equation}

The initial values for (\ref{de}) are $x_{0}$ and $x_{1}=f_{0}(x_{0},y_{0}).$
In the corresponding orbit $\{(x_{n},y_{n})\}$ of (\ref{s11})-(\ref{s12}) with
the initial point $(x_{0},y_{0})$, $x_{n}$ is determined as the solution of
(\ref{de}) together with $y_{n}$ found either using (\ref{yn}) or via
(\ref{s12}).

The right-hand side of (\ref{de}) defines the sequence of functions
\begin{equation}
F_{n}(u_{1},u_{2})=f_{n-1}(u_{1},g_{n-2}(u_{2},\sigma_{n-2}(u_{1},u_{2})))
\label{Fg}%
\end{equation}

\noindent\ on $[0,\infty)^{2}$ for $n\geq2.$ The next hypothesis allows one
possible application of Theorem \ref{sc}.

\begin{quote}
\textbf{H5}\textit{. There exist continuous functions }$\bar{f},\bar
{g}:[0,\infty)\rightarrow\lbrack0,\infty)$\textit{ such that (i) }$f_{n}%
(u_{1},u_{2})\leq\bar{f}(u_{2})$\textit{and }$g_{n}(u_{1},u_{2})\leq\bar
{g}(u_{1})$ \textit{for all} $u_{1},u_{2},n\geq0$\textit{ (ii) }$\bar{f}$ is
non-decreasing\textit{, (iii) there is }$\alpha\in(0,\infty]$\textit{ such
that }$\bar{f}(\bar{g}(u))<u$\textit{ for}$\ u\in(0,\alpha)$.
\end{quote}

The following is now implied by Theorem \ref{sc} with the composition $\bar
{f}\circ\bar{g}$ serving as the function $g$ in the theorem.

\begin{corollary}
\label{syst}Assume that the system (\ref{s11})-(\ref{s12}) satisfies
(\ref{o}), H4 and H5. If $\{(x_{n},y_{n})\}$ is an orbit of the system and
there is $n_{0}\geq1$ such that $x_{n_{0}}\in(0,\alpha)$ and $n_{0}$ is even
(odd) then the coordinates $x_{n}$ form a solution $\{x_{n}\}$ of (\ref{de})
whose even (respectively, odd) indexed terms converge monotonically to 0 for
$n\geq n_{0}$. The behavior of the sequence of components $y_{n}$ is
determined by (\ref{yn}) or directly from (\ref{s11})-(\ref{s12}).
\end{corollary}

\begin{proof}
If $F_{n}(u_{1},u_{2})$ is defined by (\ref{Fg}) then (i) and (ii) in H5 imply
that%
\[
F_{n}(u_{1},u_{2})\leq\bar{f}(g_{n-2}(u_{2},\sigma_{n-2}(u_{1},u_{2}%
)))\leq\bar{f}(\bar{g}(u_{2}))
\]

The proof is now completed by applying Theorem \ref{sc}.
\end{proof}

\medskip

Corollary \ref{syst} applies to the system (\ref{aj1})-(\ref{aj2}) without the
need to derive (\ref{aj}) by folding the system. If we define $f_{n}%
(u_{1},u_{2})=s_{n}u_{2}$ and $g_{n}(u_{1},u_{2})=u_{1}^{\lambda}%
e^{r_{n}-u_{1}-t_{n}u_{2}}$ then $\bar{f}(u)=u$ and $\bar{g}(u)=u^{\lambda
}e^{r-u}$ where $r=\sup r_{n}.$ Corollary \ref{syst} is applicable because
$f_{n},g_{n}$ and $\bar{f}\circ\bar{g}=\bar{g}$ satisfy H4 and H5. Once
$x_{n}$ is known we also obtain $y_{n}=x_{n+1}/s_{n};$ thus, if e.g.
$x_{2j+1}\rightarrow0$ (and $s_{2j}\not \rightarrow 0$) then $y_{2j}%
\rightarrow0.$ Further, since $s_{n}\leq1$ for all $n$ we see that if
$x_{2j+1}\not \rightarrow 0$ then $y_{2j}\not \rightarrow 0.$

\medskip

Theorem \ref{sc} also applies when H5 does not hold but the following does.

\begin{quote}
\textbf{H6}\textit{. There exists a continuous function }$\bar{f}%
:[0,\infty)\rightarrow\lbrack0,\infty)$\textit{ such that }$f_{n}(u_{1}%
,u_{2})\leq\bar{f}(u_{1})$\textit{ for all }$u_{1},u_{2},n\geq0$ \textit{and
there is }$\alpha\in(0,\infty]$\textit{ such that }$\bar{f}(u)<u$\textit{
for}$\ u\in(0,\alpha)$.
\end{quote}

If H6 holds then (\ref{Fg}) implies that $F_{n}(u_{1},u_{2})\leq\bar{f}%
(u_{1})$ so $\bar{f}$ may serve as the function $g$ in H1-H3. The following is
now a straightforward consequence of Theorem \ref{sc}.

\begin{corollary}
\label{syst0}Assume that the system (\ref{s11})-(\ref{s12}) satisfies
(\ref{o}), H4 and H6. If $\{(x_{n},y_{n})\}$ is an orbit of the system and
there is $n_{0}\geq1$ such that $x_{n_{0}}\in(0,\alpha)$ then the coordinates
$x_{n}$ form a solution $\{x_{n}\}$ of (\ref{de}) that converges monotonically
to 0 for $n\geq n_{0}$. The behavior of the sequence of components $y_{n}$ is
determined by (\ref{yn}) or directly from (\ref{s11})-(\ref{s12}).
\end{corollary}

An essential difference between the conclusions of corollaries \ref{syst} and
\ref{syst0} is the absence of oscillations in the latter. Corollary
\ref{syst0} corresponds to $k=1$ in Theorem \ref{sc} while Corollary
\ref{syst} to $k=2.$

Like Corollary \ref{syst}, in applying Corollary \ref{syst0} folding
calculations are not required. To illustrate, consider%
\begin{align}
x_{n+1}  &  =\frac{r_{1,n}x_{n}^{\delta_{1}}}{a_{1,n}+x_{n}^{\delta_{1}%
}+b_{1,n}y_{n}^{\delta_{3}}}\label{rbh1}\\
y_{n+1}  &  =\frac{r_{2,n}y_{n}^{\delta_{2}}}{a_{2,n}+y_{n}^{\delta_{2}%
}+b_{2,n}x_{n}^{\delta_{4}}} \label{rbh2}%
\end{align}
where we assume that%
\begin{equation}
b_{1,n},b_{2,n}\geq0,\quad a_{1,n},a_{2,n},r_{1,n},r_{2,n},\delta_{3}%
,\delta_{4}>0,\quad\delta_{1},\delta_{2}>1. \label{kp}%
\end{equation}

The autonomous version of (\ref{rbh1})-(\ref{rbh2}) with constant parameters
is introduced in \cite{Kang} as a generalized Ricker-Beverton-Holt type model
of competition for two species. Let
\[
a_{i}=\inf_{n\geq0}a_{i,n}>0,\quad r_{i}=\sup_{n\geq0}r_{i,n}<\infty,\quad
i=1,2.
\]
and
\[
\bar{f}(u_{1})=\frac{r_{1}u_{1}^{\delta_{1}}}{a_{1}+u_{1}^{\delta_{1}}}%
\]

Then (\ref{o}) and H6 hold. Setting $\bar{f}(u)<u$ gives%
\begin{equation}
\frac{r_{1}u^{\delta_{1}}}{a_{1}+u^{\delta_{1}}}<u\quad\text{or\quad}%
u^{\delta_{1}}-r_{1}u^{\delta_{1}-1}+a_{1}>0 \label{in1}%
\end{equation}

If (\ref{kp}) holds then this inequality is true for sufficiently small values
of $u$ so there is $\alpha>0$ for which the conclusion of Corollary
\ref{syst0} holds. Further, it is possible that $\alpha=\infty$ for some
parameter values; e.g. with
\[
\delta_{1}=2,\quad r_{1}^{2}<4a_{1}%
\]
the inequality in (\ref{in1}) holds for all $u\geq0$ so $x_{n}$ converges to 0
globally. In this case, $y_{n+1}\approx(r_{2}y_{n}^{\delta_{2}})/(a_{2}%
+y_{n}^{\delta_{2}})$ for all large $n$. If also $\delta_{2}=2$ and $r_{2}%
^{2}<4a_{2}$ then $y_{n}$ converges to 0 globally as well.

If $r_{1}^{2}\geq4a_{1}$ with $\delta_{1}=2$ then the quadratic inequality
$u^{2}-r_{1}u+a_{1}>0$ holds for $u\in(0,\alpha)$ where
\[
\alpha=\frac{1}{2}\left(  r_{1}-\sqrt{r_{1}^{2}-4a_{1}}\right)
\]

So if $x_{0}\in(0,\alpha)$ then $x_{n}\rightarrow0$ monotonically. This value
of $\alpha$ is sharp in the sense that it is the coordinate of an Allee fixed
point (when all parameters are constants with $\delta_{1}=2$); see \cite{Kang}
for further details about the behavior of solutions of (\ref{rbh1}%
)-(\ref{rbh2}).

\medskip

\begin{remark}
It may have been noticed that Corollary \ref{syst} involves a case where the
coordinate functions in the system are mixed to a greater degree than what is
seen in Corollary \ref{syst0}. For instance, consider the system%
\begin{align*}
x_{n+1}  &  =\frac{r_{1,n}y_{n}^{\delta_{1}}}{a_{1,n}+y_{n}^{\delta_{1}%
}+b_{1,n}x_{n}^{\delta_{3}}}\\
y_{n+1}  &  =\frac{r_{2,n}x_{n}^{\delta_{2}}}{a_{2,n}+x_{n}^{\delta_{2}%
}+b_{2,n}y_{n}^{\delta_{4}}}%
\end{align*}
which is obtained by switching $x_{n}$ and $y_{n}$ in (\ref{rbh1}%
)-(\ref{rbh2}). Straightforward calculations show that in this case Corollary
\ref{syst} may be applied, but not Corollary \ref{syst0}. On the other hand,
Corollary \ref{syst} does \textit{not} apply to (\ref{rbh1})-(\ref{rbh2})
where H5 does not hold.
\end{remark}

\medskip

\section{Conclusion and future directions}

Theorem \ref{sc} and its corollaries are quite general but they are
incomplete or not best possible. The mechanism that forces part of a solution 
to converge is not yet fully understood and a better understanding
of this mechanism may motivate research of possible future interest.

For instance, we showed that the nature of the delay in a higher order difference equation plays
a decisive role in which subsequences converge. However, the explanation of this issue
that is presented here is far from complete. 

Another issue of potential interest in the context of biological populations 
involves the connection to the strong Allee effect that was mentioned in 
the discussion after Corollary \ref{R}. For more details on this issue we
refer to \cite{LS1}.

For planar systems we obtained conditions on the system itself that imply the 
convergence of subsequences of orbits in the positive quadrant without the need 
to explicitly fold the system into a second-order equation. Thus, these conditions
simplified calculations.

Systems in higher dimensions also fold to scalar equations so Theorem \ref{sc} may be
applied directly once the system is folded. For example, consider the three dimensional 
system%
\begin{align*}
x_{n+1}  &  =e^{a_{n}-bx_{n}-cy_{n}-dz_{n}}\\
y_{n+1}  &  =p_{n}x_{n}+qz_{n}-r\ln z_{n}\\
z_{n+1}  &  =sx_{n}%
\end{align*}
where $a_{n},p_{n}\in(-\infty,\infty)$, $b,d\geq0$ and $c,q,r,s,x_{0},z_{0}%
>0$. By straightforward calculation:%
\begin{align*}
x_{n+3}  &  =e^{a_{n+2}-bx_{n+2}-cy_{n+2}-dz_{n+2}}\\
&  =e^{a_{n+2}-bx_{n+2}-ds_{n+1}x_{n+1}-c(p_{n+1}x_{n+1}+qsx_{n}-r\ln sx_{n}%
)}\\
&  =x_{n}^{cr}e^{a_{n+2}+cr\ln s-bx_{n+2}-(cp_{n+1}+ds)x_{n+1}-cqsx_{n}}%
\end{align*}

Shifting indices, the above third-order difference equation may be written as%
\begin{equation}
x_{n}=x_{n-3}^{cr}e^{a_{n-1}+cr\ln s-bx_{n-1}-(cp_{n-2}+ds)x_{n-2}-cqsx_{n-3}}
\label{spm3}%
\end{equation}

This equation is a special case of (\ref{spm}) so we may apply Corollary
\ref{R} in the study of the behavior of its solutions. 

In general, systems in three or higher dimensions are not easy to fold,
and even when folded, the resulting higher order equation may be long and 
complicated; see \cite{HSF}. Thus, developing higher dimensional
analogs of Corollaries \ref{syst} and \ref{syst0} based on suitable modifications 
of hypotheses H5 or H6 may simplify the analysis considerably.

Other issues of possible future interest include a classification of maps
$F_{n}$ for which there is a function $g$ of the type in hypotheses H1-H3.
Also, if possible, weakening some of these three hypotheses to allow for
greater flexibility would be desirable. Finally, it would be of interest to
identify applications of the results of this paper to discrete models outside
population biology.



\begin{thebibliography}{99}                                                                                               %


\bibitem {AJ}Ackleh, A.S. and Jang, S.R.-J., A discrete two-stage population
model: continuous versus seasonal reproduction, (2007) \textit{J. Difference
Eq. Appl}. \textbf{13}, 261-274

\bibitem {Cour}Courchamp, F., Berec, L. and Gascoigne, J., \textit{Allee
Effects in Ecology and Conservation}, (2008) Oxford University Press, Oxford

\bibitem {Cush}Cushing, J.M., Oscillations in age-structured population models
with an Allee effect, (1994) \textit{J. Comput. Appl. Math.}, \textbf{52}, 71-80

\bibitem {Cush2}Cushing, J.M., A juvenile-adult model with periodic vital
rates, (2006) \textit{J. Math Biol,,} \textbf{53}, 520-539

\bibitem {ES}Elaydi, S. N. and Sacker, R. J., Population models with Allee
effects: A new model, (2010) \textit{J. Biol. Dyn.}, \textbf{4}, 397-408

\bibitem {HKK}Harry, A.J., Kent, C.M. and Kocic, V.L., Global behavior of
solutions of a periodically forced sigmoid Beverton-Holt model, (2010)
\textit{J. Biol. Dyn.}, \textbf{6}, 212-234

\bibitem {Kang}Kang, Y., Dynamics of a generalized Ricker-Beverton-Holt
competition model subject to Allee effects, (2016) \textit{J. Difference Eq.
Appl.}, \textbf{22}, 687-723

\bibitem {LS}Lazaryan, N. and Sedaghat, H., Dynamics of planar systems that
model stage-structured populations, (2015) \textit{Discr. Dyn. Nature
Society}, Article ID 137182

\bibitem {LS1}Lazaryan, N. and Sedaghat, H., Extinction and the Allee effect
in an age structured Ricker population model with inter-stage interaction,
(submitted) arxiv.org 1702.02889

\bibitem {LP}Liz, E., Pilarczyk, P., Global dynamics in a stage-sturctured
discrete-time population model with harvesting, (2012) \textit{J. Theor.
Biol.} \textbf{297}, 148-165

\bibitem {LEO}Luis, R., Elaydi, S.N., and Oliveira, H., Non-autonomoous
periodic systems with Allee effects, (2010) \textit{J. Difference Eq. Appl.},
\textbf{16}, 1179-1196

\bibitem {HS}Sedaghat, H., Folding, cycles and chaos in planar systems, (2015)
\textit{J. Difference Eq. Appl.}, \textbf{21}, 1-15

\bibitem {HSF}Sedaghat, H., Folding difference and differential systems into
higher order equations, arXiv.org 1403.3995
\end{thebibliography}
\end{document}